\documentclass{article}
\usepackage{graphicx} \usepackage{amsmath} \usepackage{amssymb}
\usepackage{amsthm}
\usepackage{xcolor} \usepackage{appendix}
\usepackage{relsize}
\usepackage{hyperref}
\usepackage{comment}

\newtheorem{lemma}{Lemma}[section]
\newtheorem{definition}{Definition}[section]
\newtheorem{thm}[lemma]{Theorem}

\newtheorem{cor}[lemma]{Corollary}

\newtheorem*{remark}{Remark}

\newtheorem*{thm*}{Theorem}
\newtheorem*{lemma*}{Lemma}
\newtheorem*{cor*}{Corollary}
\newtheorem*{prop*}{Proposition}
\newtheorem*{definition*}{Definition}

\definecolor{darkgreen}{rgb}{0.0, 0.5, 0.0}
\definecolor{darkred}{rgb}{0.7, 0.0, 0.0}

\newcommand{\sqS}{\tilde{S}}

\newcommand{\Lmd}{\Lambda}

\newcommand{\Lt}{_{L^2}}
\newcommand{\Ltsq}{_{L^2}^2}
\newcommand{\Ltlt}{_{L^2l^2}}
\newcommand{\Ltltsq}{_{L^2l^2}^2}
\newcommand{\supp}{\mathrm{supp}}

\newcommand{\norm}{|\!|}
\newcommand{\Norm}[1]{\left|\!\left|#1\right|\!\right|}

\newcommand{\F}{\mathcal{F}}

\newcommand{\N}{{\mathbb{N}}}
\newcommand{\R}{{\mathbb{R}}}

\newcommand{\Cmplx}{{\mathbb{C}}}

\title{Translation-Invariant Behavior of General Scattering Transforms}
\author{Wojciech Czaja, Brandon Kolstoe, and David Koralov\footnote{Corresponding author - dkoralov@terpmail.umd.edu}}

\date{June 2024}

\begin{document}

\maketitle
\begin{abstract}
    The main result of our paper offers an alternative, simpler, proof of Mallat's result on the translation invariance of the limiting behavior of sequences of Wavelet Scattering Transforms, which (unlike Mallat's proof) does not rely on the admissibility condition or on the density of a logarithmic Sobolev space in $L^2$. Furthermore, this result is generalized to a broader class of scattering transforms, including, for instance, a modification of the Fourier Scattering Transform. As a result, we also prove a new upper bound for the translation contraction for the Fourier Scattering Transform.
\end{abstract}
 
\section{Introduction}
In this paper, we will deal with scattering transforms\footnote{In this paper, scattering transforms are not necessarily wavelet based. When referring to Mallat's Wavelet Scattering Transform, we will do so explicitly.} - a type of transformation used in the field of pattern recognition and classification of audio and image data. When used in practice, a scattering transform $\mathcal{S}$ is applied to images and audio data prior to training a neural network for classification. Once the neural network is trained (and a satisfactory neural network function $\mathcal{N}$ is obtained), the resulting classifier is taken to be $\mathcal{N}\circ\mathcal{S}$. Scattering transforms embed audio/image data, theoretically representable as $L^2(\R^d)$ functions, (with $d=1$ for audio signals and $d=2$ for images)\footnote{In general, scattering transforms are defined for arbitrary finite-dimensional spaces $\R^d$.} in the space of countable collections of $L^2(\mathbb{R}^d)$ functions, endowed with the $L^2l^2$ norm. When used in practice, $L^2(\R^d)$ functions are stored as either vectors or matrices of their values on a discrete lattice, and the space $L^2l^2(\R^d)$ is represented with finite sequences of such vectors/matrices.

For certain scattering transforms, including the Wavelet Scattering Transform \cite{Mallat} and the Fourier Scattering Transform \cite{CzajaLi}, it has been proved that the embedding of functions into the $L^2l^2$ space preserves their norm, and exhibits near-translation-invariant and near-diffeomorphism-invariant behavior.\footnote{These properties are not unique to the scattering transform. For example, taking the $L^2$ norm of a function also has all three of those properties. However, the $L^2$-norm-operator ($\Norm{\cdot}_{L^2}: L^2(\R^n)\rightarrow \R$) loses an enormous amount of valuable information on natural images that scattering transforms preserve. It should be noted, nonetheless, that scattering transforms are not perfectly injective.} In practice, these properties have shown themselves to be very useful in audio and image processing problems, as this allows a neural network to spend less resources on learning redundant patterns, and instead learn complex patterns more efficiently with less training data.

The focus of this paper (Theorem \ref{MAIN RESULT - UPPER BOUND} and Corollary \ref{Main Result main_result (FINAL)}) is demonstrating that Mallat's proof (in \cite{Mallat}) of translation-invariant behavior of the Wavelet Scattering Transform\footnote{The proof that the norm of the difference of $S_J[P_J] f(\cdot)$ and $S_J[P_J] f(\cdot+c)$ goes to $0$ as $J\rightarrow\infty$.} can be replaced with a  simpler proof, which does not require the intricate ``admissibility condition" (see \cite{Mallat}) and is applicable to a broader range of scattering transforms\footnote{Including scattering transforms that are not based on wavelet frames.}. Our main motivation for writing the paper was to prove Corollary \ref{Main Result main_result (FINAL)} for coherent sequences of Fourier Scattering Transforms, defined in \cite{CzajaLi}.

\section{Preliminaries}

\subsection{Definition of the Wavelet Scattering Transform}
Let us recall the definition of the Wavelet Scattering Transform, introduced by Mallat in \cite{Mallat}.

\begin{definition}\textbf{(Wavelets)}
        Let $G$ be a finite group of rotations on $\mathbb{R}^d$, together with reflection about the origin. Let $\psi:\mathbb{R}^d\rightarrow\Cmplx$ be a $L^1\cap L^2$ function such that $\widehat\psi$ has compact support, and 
    \begin{equation}\label{Wavelet_frame_property_motheronly}
        \sum_{r\in G}\sum_{j=-\infty}^{\infty}\left|\widehat\psi(2^{j}r\xi)\right|^2=1 \textrm{ a.e.}
    \end{equation}
    \noindent Furthermore, let $\phi$ be an $L^2$ function such that 
    \begin{equation}\label{father_wavelet_def}
        \sum_{r\in G}\sum_{j=-\infty}^{0}\left|\widehat\psi(2^{j}r\xi)\right|^2=\left|\widehat{\phi}(\xi)\right|^2\textrm{ a.e.}
    \end{equation}
    
    \noindent The function $\psi$ is called the mother wavelet, and the function $\phi$ is called the father wavelet.
\end{definition}

 \noindent It follows immediately from (\ref{father_wavelet_def}) that, given a mother wavelet $\psi$ that satisfies
(\ref{Wavelet_frame_property_motheronly}), the absolute value of the Fourier transform of the father wavelet, $|\widehat\phi|$, is determined uniquely.

\noindent Let us denote $\phi_{2^J}(\xi):=2^{-dJ}\phi(2^{-J}\xi)$ and $\psi_{2^j,r}(x):=2^{dj}\psi(2^jr^{-1}x).$
Let $\Lmd_J$ denote the set $$\Lmd_J=\left\{2^jr|j>-J\right\},\textrm{and let }\Lmd_\infty:=\bigcup_{J=-\infty}^{\infty} \Lmd_J.$$
For $J\in\mathbb{Z}\cup\{\infty\}$, let $${P}_J:=\bigcup_{k=0}^{\infty} \Lmd_J^k,$$ where $$\Lmd_J^k:=\Lmd_J\times....\times\Lmd_J\textrm{ (k times).}$$

\begin{definition} \textbf{(Scattering Propagator for Wavelets \cite{Mallat})}\\For each $p\in\Lmd_\infty^k, p=(\lambda_1,\lambda_2,...,\lambda_k)$, inductively define the scattering propagator \begin{equation*}
    U[p]f:=
    \begin{cases}
        f & \text{if k=0,}\\
        \mathlarger|U[\lambda_1,\lambda_2,...,\lambda_{k-1}]f*\psi_{\lambda_k}\mathlarger| & \text{if $k\geq 1$,}
    \end{cases}
\end{equation*}
for functions $f\in L^2.$
    
\end{definition}
\begin{definition}
    For any $R\subseteq P_\infty$, we define
    $$U[R]f:=\{U[p]f,p\in R\}$$
    and
    $$S_J[R]f:=\{U[p]f*\phi_{2^J},p\in R\}$$
    
\end{definition}
\begin{definition}\textbf{(Wavelet Scattering Transform \cite{Mallat})}\\
    The (windowed) Wavelet Scattering Transform of a function $f\in L^2$ is defined as 
$$S_J[{P}_J]f:=\{U[p]f*\phi_{2^J},p\in {P}_J\}.$$
\end{definition}

\begin{remark}
    The ordinary wavelet transform is denoted by $W[\Lmd_{\infty}]$. Thus, $$W[\Lmd_\infty]f:=\left\{f*\psi_{\lambda}\textrm{ s.t. }\lambda\in\Lmd_\infty\right\}.$$
    The operator of convolution with $\phi_{2^J}$ is denoted by $A_J$. Thus,
    $$A_Jf:=f*\phi_{2^J}.$$
    Since the ordinary wavelet transform, the absolute value operator, and the operator of convolution with $\phi_{2^J}$ are all Lipschitz with Lipschitz constant less than or equal to $1$, one can conclude that the Wavelet Scattering Transform is also non-expansive (i.e., Lipschitz with constant $\leq$ 1). 
\end{remark}

\begin{remark}
    We denote
    $$U_J[\Lambda_J]f:=\left\{A_Jf,\left|f*\psi_{\lambda}\right|{\textrm{ s.t. }}\lambda\in\Lmd_{J}\right\}.$$
\end{remark}

\subsection{Properties of the Wavelet Scattering Transform}
In \cite{Mallat}, Mallat proved that, for wavelets satisfying a certain admissibility condition\footnote{Although there has been some research (see \cite{Waldspurger}, for example) attempting to relax the admissibility condition or find an example of an admissible wavelet, we are not aware of any results demonstrating an example of an admissible wavelet for dimensions $d\geq 2$.},
$$\norm S_J[{P}_J]f\norm\Ltlt=\norm f\norm_{L^2}.$$

We will include the exact statement of Mallat's theorem below.

\begin{definition}\label{Admissibility_condition}\textbf{(Admissibility Condition \cite{Mallat})}\\
    A mother wavelet $\psi$ is said to be admissible if there are $\eta\in \mathbb{R}^d$ and $\rho\geq 0$ such that 
    $|\widehat\rho(\xi)|\leq|\widehat\phi(2\xi)|$ and $\widehat\rho(0)=1,$ such that the function 
    $$\widehat\Psi(\xi):=|\widehat\rho(\xi-\eta)|^2-\sum_{k=1}^\infty k(1-|\widehat\rho(2^{-k}(\xi-\eta))|^2)$$
    satisfies
    $$\alpha:=\inf_{1\leq|\xi|\leq 2}\sum_{j=-\infty}^{\infty}\sum_{r\in G} \widehat\Psi(2^{-j}r^{-1}\xi)|\widehat\psi(2^{-j}r^{-1}\xi)|^2> 0.$$
\end{definition}

\begin{thm} \textbf{(Norm Preservation \cite{Mallat})}
    Suppose that an admissible wavelet $\psi$ satisfies (\ref{Wavelet_frame_property_motheronly}) and (\ref{father_wavelet_def}). Then, for each $f\in L^2(\mathbb{R}^d),$ we have that
    $$\lim_{m\rightarrow\infty}\norm U[\Lambda_J^m]f\norm\Ltltsq=\lim_{m\rightarrow\infty}\sum_{n=m}^\infty\norm S_{J}[\Lambda_J^n]f\norm\Ltltsq=0$$
    and
    $$\norm S_J[{P}_J]f\norm\Ltlt = \norm f\norm\Lt.$$
\end{thm}

The difficult part of this theorem is that the first limit is equal to zero, with the rest then following from this and the non-expansive property. Mallat also proved the following result about the behavior of the scattering transform as the parameter $J$ increases.

\begin{thm}\textbf{(Limiting Translation Invariance \cite{Mallat})}\label{Theorem 2.10 Mallat}
    For admissible scattering wavelets, for each $f\in L^2(\mathbb{R}^d)$ and $c\in \R^d$,
    $$\lim_{J\rightarrow\infty}\Norm{ S_J[P_J]f- S_J[P_J]T_cf}\Ltlt =0,$$
    where $T_cf:=f(\cdot-c).$
\end{thm}

 \subsection{The General Scattering Transform}

Scattering Transforms, in their most general sense, were defined by T. Wiatowski and H. B{\"o}lcskei in \cite{Wiatowski}. For our purposes, we will use a slightly more restricted, but still very general definition. In particular, our definition of a scattering transform (a mapping $S:L^2(\R^d)\rightarrow L^2l^2(\R^d)$) will be satisfied if $S$ satisfies Definition 3 in \cite{Wiatowski} for a frame collection $\Psi$ with the same frame at every layer (i.e., if $\Psi_i = \Psi_j $ for all $ i,j$), and if the upper frame bound for this frame is $\leq 1$. Unlike \cite{Wiatowski}, we do not require the existence of a lower frame bound. Let us introduce the notion of a scattering transform more precisely.

    \begin{definition}\label{semi-discrete Bessel Sequence}\textbf{(Semi-Discrete Bessel Sequence)}
        A Semi-Discrete Bessel Sequence\footnote{If the expression in formula (\ref{General_scattering_upper_frame_bound}) has a lower bound $A>0$, then the semi-discrete Bessel Sequence becomes a semi-discrete frame.} is a sequence $\mathcal{G}$ of $L^1\cap L^2$ functions such that    \begin{equation}\label{General_scattering_upper_frame_bound}
            \sum_{g\in\mathcal{G}} \left|\hat{g}(\xi)\right|^2\leq 1 \; a.e.
        \end{equation}\label{frame_property_general}

    \end{definition}
    \begin{definition}\textbf{(Index Set and output-generating function)}
        Suppose that for a semi-discrete Bessel Sequence $\mathcal{G}$ and a set $\Lmd$ (with $0\not\in\Lmd$)  we have$$\mathcal{G}=\{g_\lambda\}_{\lambda\in\Lmd}\cup\{g_0\}.$$ Then, we say that $\mathcal{G}$ has the index set $\Lmd$ and the output-generating function $g_0$. We refer to the set $\{g_\lambda\}_{\lambda\in\Lmd}$ as the set of peripheral elements of $\mathcal{G}$.
    \end{definition}

    \begin{definition} \textbf{(Scattering Propagator)}Let $\mathcal{G}$ be a semi-discrete Bessel Sequence with index set $\Lambda$. For each $p\in\Lambda^k, p=(\lambda_1,\lambda_2,...,\lambda_k)$, we inductively define the scattering propagator 
    \begin{equation*}
        \tilde{U}[p]f:=
        \begin{cases}
            f & \text{if }k=0,\\
            \left|\tilde{U}[\lambda_1,\lambda_2,...,\lambda_{k-1}]f*g_{\lambda_k}\right| & \text{if } k\geq 1,
        \end{cases}
    \end{equation*}
    for functions $f\in L^2.$
        
    \end{definition}

\begin{thm}
    $\tilde{U}$ is a non-expansive operator $L^2\rightarrow L^2l^2$.
\end{thm}
The proof of why this is true is identical to Mallat's proof in \cite{Mallat} of why $U_J$ is non-expansive. This theoerm is also proved in \cite{WiatowskiBIG} and \cite{Balan}.

\begin{definition}\label{General Scattering Transform}
    Let $\mathcal{G}$ be a semi-discrete Bessel Sequence with index set $\Lambda$ and output-generating function $g_0$. Let $P:=\bigcup_{i=0}^{\infty} \Lambda^i$.
    The General Scattering Transform of a function $f\in L^2$ is defined as 
    $$\tilde{S}[{P}]f:=\{\tilde{U}[p]f*g_0,p\in P\}.$$
\end{definition}

The proof of the following theorem is also similar to the proof of the corresponding results in \cite{Mallat} and is provided in \cite{Wiatowski} and \cite{WiatowskiBIG}.
\begin{thm}\label{sqSJ nonexpansive}  Let $\mathcal{G}$ be a semi-discrete Bessel Sequence with index set $\Lambda$ and output-generating function $g_0$. Then ${\tilde{S}[P]}$ is a non-expansive operator.
In particular,
    $$\Norm{\tilde{S}[P]f}\Ltlt\leq\Norm{f}\Lt.$$

\end{thm}

\subsection{Sequences of General Scattering Transforms}

In \cite{Mallat}, the near-translation invariance is interpreted in terms of the limiting behavior of $S_J[P_J]$ as $J\rightarrow\infty$. Namely, Mallat considers a sequence of scattering transforms depending on a parameter $J$ that defines the size of the central element $\phi_{2^J}$. We introduce a corresponding notion for general scattering transforms.

\begin{definition}
    Let $\{\mathcal{G}_J\}_{J\in\mathbb{N}}$ be a sequence of semi-discrete Bessel Sequences with index sets $\Lmd_J$ and output-generating functions $g_{0,J}.$ Suppose that $\hat{g}_{0,J}$ is compactly supported for each $J$. We call the sequence of scattering transforms $\tilde{S}_J$ corresponding to the semi-discrete Bessel Sequences $\mathcal{G}_J,J\in\mathbb{N}$, ``coherent" if
    \begin{enumerate}
        \item $\mathcal{G}_J\setminus\{g_{0,J}\}\subseteq \mathcal{G}_{J+1}\setminus\{g_{0,{J+1}}\}$ for each $J$.
                \item $\lim_{J\rightarrow\infty} \sup(\{|\xi|\textrm{ s.t. }\hat{g}_{0,J}(\xi)\neq 0\})=0$.
    \end{enumerate}
\end{definition}
\begin{remark}
 The index sets $\Lmd_J$ for different $J$ are selected to be consistent with each other, in particular, $\Lmd_J\subseteq\Lmd_{J+1}$. The only element of $\mathcal{G}_J$ that might not be in $\mathcal{G}_{J+1}$ is the output-generating function $g_{0,J}$.
\end{remark}

We denote $D_J:=\sup(\{|\xi|\textrm{ s.t. }\hat{g}_{0,J}(\xi)\neq 0\}),~ P_J:=\bigcup_{n=0}^\infty \Lmd_J^n$.

\begin{remark}
    The sequence of frames $\mathcal{G}_J$ is unrelated to the frame collections in Definition 2 of \cite{Wiatowski}. Whereas B{\"o}lcskei and Wiatowski apply different frames to different layers of the same scattering transform, we, instead apply the same frame to each layer of each individual scattering transform, and analyze the behavior of a sequence of scattering transforms, each corresponding to a different frame. For example, the translation-invariant behavior in the setting of the Fourier Scattering Transform that we discuss in Section \ref{Application - Sequences of Fourier Scattering Transforms} can only be described with a sequence of transformations, corresponding to a sequence of semi-discrete frames.
\end{remark}

 \section{Our Result}

Recall that Mallat's result, Theorem \ref{Theorem 2.10 Mallat}, states that, for wavelets satisfying the admissibility condition (Definition \ref{Admissibility_condition}), for each $f\in L^2,$ $$\lim_{J\rightarrow\infty}\Norm{S_J[P_J]f-S_J[P_J]T_cf}\Ltlt\rightarrow0,$$ where $(T_cf)(x):=f(x-c)$.

We will now state and prove the corresponding result for the class of all coherent scattering transform sequences (including Mallat's Wavelet Scattering Transforms), without requiring the admissibility condition.

\begin{thm}\label{MAIN RESULT - UPPER BOUND}\footnote{Note that, since $\tilde{S}[P]$ is non-expansive, $\norm \tilde{S}[P]f- \tilde{S}[P]T_cf \norm\Ltlt$ will also be bounded from above by $2\Norm{f}\Lt$, which, in different situations, may be an either better or worse bound than the one provided in this theorem. The bound provided by the theorem is sharper when $D|c|$ is small.}
    Let $\tilde{S}$ be a scattering transform with an index set $\Lmd$ and an output-generating function $g_0$. Let $D=\sup(\{|\xi|\textrm{ s.t. }\hat{g}_0(\xi)\neq 0\})<\infty$. As before, let $P:=\cup_{n=0}^{\infty} \Lmd^n$. Then, for each $c\in\mathbb{R}^d$, 
    $$\Norm{ \tilde{S}[P]f- \tilde{S}[P]T_cf}\Ltlt \leq  2\pi D|c|\Norm{f}\Lt.$$
\end{thm}

\begin{proof}
    Observe that
 
     $$\Norm {\tilde{S}[P] f-\tilde{S}[P] T_c f} \Ltltsq=\sum_{p\in{P}} \Norm{ U[p]f*g_0 - U[p](T_cf)*g_0}\Ltsq.$$

     Denote $U[p]f$ as $u_p$. Since the translation operator $T_c$ commutes both with the absolute value operator and with the convolution with $g_\lambda$  for each $\lambda$, we have $U[p](T_cf)=T_cu_p$.
    Without loss of generality, assume that $\norm f\norm\Lt\neq0.$
    
     By the Plancherel identity, for each $p$, we have $$\norm u_p*g_0-(T_cu)*g_0\norm\Ltsq=\norm \widehat{u}_p(\xi)\widehat g_0(\xi)-e^{-2\pi i\xi\cdot c}\widehat{u}_p(\xi)\widehat g_0(\xi)\norm\Ltsq.$$
     Next, notice that for any real number $r$, $ |e^{ir}-1|\leq|r|.$
     Hence, for each $\xi\in B(D,0)$, since $|\xi\cdot c|\leq D|c|$,$$ |e^{-2\pi i\xi\cdot c}-1|<2\pi D| c|.$$
     Thus, since $\textrm{supp}(\widehat{g}_0)\subseteq B(D,0)$, we have
     $$\Norm{ \widehat{u}_p(\xi)\widehat g_0(\xi)-e^{-2\pi i\xi\cdot c}\widehat{u}_p(\xi)\widehat g_0(\xi)}\Ltsq=\Norm{ (1-e^{-2\pi i\xi\cdot c})\widehat{u}_p(\xi)\widehat g_0(\xi)}\Ltsq\leq$$
     $$\leq (2\pi D|c|)^2{\Norm{\widehat{u}_p\widehat g_0}}\Ltsq.$$
     Moreover, since our scattering transforms are non-expansive, we have 
     $$\norm \tilde{S}[P] f\norm\Ltltsq=\sum_{p\in{P}} \norm u_p*g_0\norm\Ltsq=\sum_{p\in{P}} \norm \hat{u}_p\widehat{g}_0\norm\Ltsq\leq\norm f\norm\Ltsq.$$
     Hence, we have 
     $$\norm \sqS[P] f-\sqS[P] T_c f\norm \Ltltsq=\sum_{p\in{P}}\norm \widehat{u}_p(\xi)\widehat g_0(\xi)-e^{-2\pi i\xi\cdot c}\widehat{u}_p(\xi)\widehat g_0(\xi)\norm\Ltsq\leq$$
     $$\leq(2\pi D |c|)^2\sum_{p\in P} \norm \hat{u}_p\widehat{g}_0\norm\Ltsq=(2\pi D |c|)^2\Norm{\tilde{S}[P]f}\Ltltsq\leq(2\pi D|c|)^2\Norm{f}\Ltltsq.$$
\end{proof}

\begin{cor}\label{Main Result main_result (FINAL)}
    For any coherent sequence of scattering transforms $\tilde{S}_J[P_J]$, for each $ f\in L^2(\mathbb{R}^d),$ and each $ c\in\mathbb{R}^d$,
    $$\lim_{J\rightarrow\infty}\Norm{ \tilde{S}_J[P_J]f- \tilde{S}_J[P_J]T_cf}\Ltlt =0.$$
\end{cor}
\begin{proof}\
    This is obtained directly by applying the bound from Theorem \ref{MAIN RESULT - UPPER BOUND}, which states that for each $ J,$
    $$\Norm{ \tilde{S}_J[P_J]f- \tilde{S}_J[P_J]T_cf}\Ltltsq \leq (2\pi D_J |c|)^2\Norm{f}\Ltltsq,$$
    and noticing that $D_J\rightarrow 0$, while all the other factors on the right-hand side do not depend on $J$.
\end{proof}

 \section{Application - Sequences of Fourier Scattering Transforms}\label{Application - Sequences of Fourier Scattering Transforms}
    As mentioned earlier, our main motivation for writing the paper was to prove Corollary \ref{Main Result main_result (FINAL)} for coherent sequences of Fourier Scattering Transforms \cite{CzajaLi}. Like the Wavelet Scattering Transform, the Fourier Scattering Transform is a special case of the Scattering Transform given in Definition \ref{General Scattering Transform}. A major benefit of the Fourier Scattering Transform is that the norm-preservation property is obtained without the need for an admissibility condition. Unfortunately, in the case of the Fourier Scattering Transform, the result on near-diffeomorphism-invariance only holds for nearly band-limited functions in the input. Another limitation of the Fourier Scattering Transform was that, prior to this paper, there has not been a result on the translation-invariant behavior analogous to that of \cite{Mallat}. Thanks to \ref{Main Result main_result (FINAL)}, we have bridged this gap.
    
    We state the definitions of a Uniform Covering Frame and the Fourier Scattering Transform formally below.

    \begin{definition}\textbf{(Uniform Covering Frame \cite{CzajaLi})}
	Suppose that we have a Semi-Discrete Bessel Sequence
	\[
	\mathcal{F}=\{g_0\}\cup\{g_\lambda\colon \lambda\in\Lmd\}.
	\]
	$\mathcal{F}$ is said to be a uniform covering frame if it satisfies the following assumptions. 
	\begin{enumerate}
		\item
		\emph{Assumptions on $g_0$ and $g_\lambda$}. Let $g_0\in L^1(\R^d)\cap L^2(\R^d)$ be such that $\hat{g_0}$ is supported in a neighborhood of the origin and $|\hat{g_0}(0)|=1$. For each $\lambda\in\Lmd$, let $g_\lambda\in L^1(\R^d)\cap L^2(\R^d)$ be such that $\supp(\hat{g_\lambda})$ is compact and connected. 
		\item
		\emph{Uniform covering property}. There exists an $R$ such that, for each $ \lambda$, there exists $ \xi_\lambda\in\R^d$ such that $\supp(\hat{g}_\lambda)\subseteq B(R,\xi_\lambda)$, where $B(a,b)$ denotes the open ball of radius $a$ around point $b$. 
		\item
		\emph{Frame condition}. Assume that, for all $\xi\in\R^d$,
		\begin{equation}
		\label{eq frame} 
		|\hat{g_0}(\xi)|^2+\sum_{\lambda\in\Lmd} |\hat{g_\lambda}(\xi)|^2=1.
		\end{equation}
	\end{enumerate}
    \end{definition}
    This implies that $\mathcal{F}$ is a semi-discrete Parseval frame for $L^2(\R^d)$: For all $f\in L^2(\R^d)$,
		\[
		\|f*g_0\|_{L^2}^2+\sum_{\lambda\in\Lmd} \|f*g_\lambda\|_{L^2}^2 = \|f\|_{L^2}^2.
		\]
    \begin{definition}\textbf{(Fourier Scattering Transform \cite{CzajaLi})}
        A Scattering transform (see Definition \ref{General Scattering Transform}) $\tilde{S}$ associated with a uniform covering frame $\mathcal{F}$ is called a Fourier Scattering Transform.
    \end{definition}
    
    We'll now recall some properties of the Fourier Scattering Transform, proved in \cite{CzajaLi}.    
    Let $\F=\{g_0\}\cup\{g_\lambda\colon \lambda\in\Lmd\}$ be a uniform covering frame, and let $\tilde{S}_\F$ be the Fourier scattering transform. It was proved in \cite{CzajaLi} that
    \textbf{$\tilde{S}_\F$ conserves energy}. In other words,
	 for each $f\in L^2(\R^d)$,
		\[
		\|\tilde{S}_\F(f)\|_{L^2\ell^2}=\|f\|_{L^2}.
		\]
    Moreover, \textbf{$\tilde{S}_\F$ is non-expansive on $L^2(\R^d)$ }. That is,    
		for each $f,g\in L^2(\R^d)$, 
		\[
		\|\tilde{S}_\F(f)-\tilde{S}_\F(g)\|_{L^2\ell^2}\leq \|f-g\|_{L^2}.
		\]
    It was also shown that \textbf{$\tilde{S}_\F$ contracts sufficiently small additive diffeomorphisms of almost band-limited functions}. Namely, let $\varepsilon\in [0,1)$ and $R>0$. There exists a universal constant $C>0$, such that, for all $(\varepsilon,R)$ band-limited $f\in L^2(\R^d)$ and all $\tau\in C^1(\R^d;\R^d)$ with $\|\nabla\tau\|_{L^\infty}\leq 1/(2d)$,
		\[
		\|\tilde{S}_\F(f)-\tilde{S}_\F(T_\tau f)\|_{L^2\ell^2}
		\leq C(R\|\tau\|_{L^\infty}+\varepsilon)\|f\|_{L^2}.
		\]
    The main trait that separates the Fourier Scattering Transform from the Wavelet Scattering Transform is the fact that the $L^2l^2$ norm of ${\tilde{U}[\Lmd^n]f}$ decays exponentially as $n\rightarrow\infty$, which automatically implies norm-preservation for $\tilde{S}_\F$ without the need for an admissibility condition. Specifically, for each uniform covering frame, there exists a constant $C_0\in(0,1)$ such that for any $f\in L^2(\R^d)$ for any $K\geq 1$,
    $$\Norm{\tilde{U}[\Lmd^K]f}\Ltltsq\leq C_0^{K-1}\left(\Norm{f}\Ltsq-\Norm{f*g_0}\Ltltsq\right).$$
    In contrast, while it was shown in \cite{Mallat} that, for the Wavelet Scattering Transform, $\Norm{{U}[\Lmd^n]f}\Ltlt\rightarrow 0$ for admissible wavelets, this convergence can be arbitrarily slow, as shown in \cite{Fuhr}.
    
    The following result holds for norm differences under translation for the Fourier Scattering Transform.
    \begin{thm}\label{FST contracts small translations} \textbf{($\tilde{S}_\F$ contracts sufficiently small translations of $L^2(\R^d)$ \cite{CzajaLi})}
		There exists $C>0$ depending only on $\F$, such that for all $f\in L^2(\R^d)$ and $y\in \R^d$, 
		\[
		\Norm{\tilde{S}_\F(f)-\tilde{S}_\F(T_c f)}\Ltlt
		\leq C|c|\|\nabla f_0\|_{L^1} \|f\|_{L^2}.
		\] 
		
    \end{thm}

    Applying the bound in Theorem \ref{MAIN RESULT - UPPER BOUND}, we can provide an alternative to the upper bound in Theorem $\ref{FST contracts small translations}$:
        $$\Norm{\tilde{S}_\F(f)-\tilde{S}_\F(T_c f)}\Ltlt\leq {2\pi D|c|\Norm{f}\Lt}.$$
    In particular, applying Corollary \ref{Main Result main_result (FINAL)} to the Fourier Scattering Transform, we obtain:
    \begin{cor}
        Suppose $\{F_J\}_{J\in\N}$ is a coherent sequence of Uniform Covering Frames, with index sets $\Lmd_J$, output-generating functions $g_{0,J}$, and with the path set $P_J:=\cup_{k=0}^\infty \Lmd_J^k$, Let $\tilde{S}_{\mathcal{F}_J}$ be the corresponding Fourier Scattering Transforms. Then, for every $ f\in L^2(\mathbb{R}^d),\textrm{ for every } c\in\mathbb{R}^d$,
        $$\lim_{J\rightarrow\infty}\Norm{ \tilde{S}_{\mathcal{F}_J}[P_J]f- \tilde{S}_{\mathcal{F}_J}[P_J]T_cf}\Ltlt =0.$$        
    \end{cor}
    \begin{proof}
        This is a special case of Corollary \ref{Main Result main_result (FINAL)}, when the underlying semi-discrete Bessel Sequences are Uniform Covering Frames.
    \end{proof}
    Thus, we have been able to show an analogous result to Mallat's limiting translation invariance (of the Wavelet Scattering Transform) for the Fourier Scattering Transform.

 \end{document}